\documentclass[]{amsart}
\usepackage[applemac]{inputenc}
\usepackage[english]{babel}
\usepackage{caption}
\usepackage{enumerate}
\usepackage{amsmath}
\usepackage{bm}
\usepackage{bbm}
\usepackage{xcolor}
\usepackage{graphicx}
\usepackage{amssymb}
\usepackage{tikz}

\usetikzlibrary{arrows,calc,shapes,decorations.pathreplacing,backgrounds}

\usepackage{color}

 \newtheorem{thm}{Theorem}[section]
 
 \newtheorem{lem}[thm]{Lemma}
 \newtheorem{prop}[thm]{Proposition}
 \theoremstyle{definition}
 \newtheorem{defn}[thm]{Definition}
 \theoremstyle{remark}
 \newtheorem{rem}[thm]{Remark}
 \newtheorem*{rem*}{Remark}
 
 \numberwithin{equation}{section}
 
\newcommand{\fxeta}{f_{x,\eta}}
\newcommand{\gxeta}{g_{x,\eta}}
\renewcommand{\epsilon}{\varepsilon}
\newcommand{\vb}{|}

\begin{document}
\title{Symmetric Excited States for a Mean-Field Model for a Nucleon}

\author[L. Le Treust]{Loïc Le Treust$^1$}
\address{$^1$Ceremade, Universit\'e Paris-Dauphine, Place de Lattre de Tassigny, 75775 Paris Cedex 16, France}
\email{letreust@ceremade.dauphine.fr}
\author[S. Rota Nodari]{Simona Rota Nodari$^2$}
\address{$^2$Laboratoire de Mathématiques, Université de Cergy-Pontoise, 2 avenue Adolphe Chauvin, 95302 Cergy-Pontoise Cedex, France}
\email{simona.rota-nodari@u-cergy.fr}

\begin{abstract}
 In this paper, we consider a stationary model for a nucleon interacting with the $\omega$ and $\sigma$ mesons in the atomic nucleus. The model is relativistic, and we study it in a nuclear physics nonrelativistic limit. By a shooting method, we prove the existence of infinitely many solutions with a given angular momentum. These solutions are ordered by the number of nodes of each component.  
\end{abstract}
\date{\today}
\maketitle
\tableofcontents
%
%
\section{Introduction}
This article is concerned with the existence of excited states for a stationary relativistic mean-field model for atomic nuclei in the nuclear physics nonrelativistic limit. To our knowledge, this model was first studied by Esteban and Rota Nodari; in two recent papers \cite{estebanrotanodari,estebanrotanodarirad}, the authors showed the existence of so-called ground states (see \cite{estebanrotanodari} for more details about the definition of ground states). 

As the authors formally derived in \cite{estebanrotanodarirad}, the equations of the model are given, in the case of a single nucleon, by 
\begin{equation}\label{eqdiracnrl}
\left\{
\begin{aligned}
&i\bm{\sigma}\cdot\nabla\chi+|\chi|^2\varphi-a|\varphi|^2\varphi+b\varphi=0,\\
&-i\bm{\sigma}\cdot\nabla\varphi+\left(1-|\varphi|^2\right)\chi=0,
\end{aligned}
\right.
\end{equation}
with $a$ and $b$ two positive parameters linked to the coupling constants and the nucleon's and mesons' masses. This system is the nuclear physics nonrelativistic limit of the $\sigma$-$\omega$ relativistic mean-field model (\cite{waleckasigmaomega, walecka}) in the case of a single nucleon. We remind that $\bm\sigma$ is the vector of Pauli matrices $(\sigma_1,\sigma_2,\sigma_3)$, and $\varphi,\chi:\mathbb R^3\to\mathbb C^2$. 

As in \cite{estebanrotanodarirad}, we look for solutions of (\ref{eqdiracnrl}) in the particular form
\begin{equation}\label{eqsolrad}
\left(\begin{aligned}\varphi(x)\\\chi(x)\end{aligned}\right)=\left(
\begin{aligned}
&g(r)\left(\begin{aligned}1\\0\end{aligned}\right)\\
&i f(r)\left(\begin{aligned}&\cos\vartheta\\&\sin\vartheta e^{i\phi}\end{aligned}\right)
\end{aligned}
\right)
\end{equation}
where $f$ and $g$ are real valued radial functions {and $(r,\vartheta,\phi)$ are the spherical coordinates of $x$}. The system (\ref{eqdiracnrl}) then turns to a nonautonomous planar differential system which is 
\begin{equation}\label{eq_nonautonome}
\left\{\begin{aligned}
f'+\frac{2}{r}f&=g(f^2-a g^2+b)\,,\\
g'&=f(1-g^2)\,.
\end{aligned}\right.
\end{equation}
In order to avoid solutions with singularities at the origin, we impose $f(0)=0$, and, since we are interested in finite energy solutions of (\ref{eqdiracnrl}), we seek solutions of \eqref{eq_nonautonome} {that are localized {\it{i.e.}} solutions} which fulfill
\begin{equation}\label{limiting}(f(r), g(r)) \longrightarrow (0,0)\quad\mbox{as}\quad  r\longrightarrow +\infty
\,.
\end{equation}

In \cite[Proposition 2.1]{estebanrotanodarirad}, Esteban and Rota Nodari showed that there is no nontrivial solution of \eqref{eq_nonautonome} such that \eqref{limiting} is satisfied unless $a-2b>0$. 
Hence, in what follows, we assume $a-2b>0$. 

For every given $x$, there exists a local solution $(f_x,g_x)$ of 

\begin{equation}\label{pbcauchy_nonautonome}
	\left\{\begin{aligned}
&f'+\frac{2}{r}f=g(f^2-a g^2+b)\,,\\
&g'=f(1-g^2)\,,\\
&f(0)=0,~g(0)=x.
\end{aligned}\right.
\end{equation}
The problem is to find $x$, such that the corresponding solution is global (i.e. defined for all $r\ge0$), and satisfies \eqref{limiting}. 

In {\cite[Proposition 2.1]{estebanrotanodarirad}}, Esteban and Rota Nodari proved that {if $(f_x,g_x)$ is a solution of \eqref{pbcauchy_nonautonome} satisfying \eqref{limiting} then $g_x^2(r)<1$, for all $r$ in $[0,+\infty)$.}  So, in particular, $x=g_x(0)$ must be chosen such that $x^2<1$. This creates additional difficulties to deal with.

Since the system of equations \eqref{eq_nonautonome} is symmetric with respect to $0$, we study the problem \eqref{pbcauchy_nonautonome} with $x\in[0,1).$ Moreover, let us remark that if $x=0$ then $(f_x,g_x)(r)=(0,0)$ for all $r\geq 0$ { is the unique solution of \eqref{pbcauchy_nonautonome}}.

In \cite{estebanrotanodarirad}, the authors proved the existence of a global localized solution $(f_x,g_x)$ of \eqref{pbcauchy_nonautonome} such that $f_x(r)<0<g_x(r)$ for all $r\in (0,+\infty)$. In this paper, we generalize this results by showing the existence of global localized solutions with any given number of nodes. Our main result is the following.

\begin{thm}\label{thm_main} Assume $a>2b>0$. There exists an increasing sequence $\{x_k\}_{k\ge 0}$ in $(0,1)$ with the following properties. For every $k\ge0$,
\begin{enumerate}
\item the solution $(f_{x_k},g_{x_k})$ of \eqref{pbcauchy_nonautonome} is a global solution;
\item { both $f_{x_k}$ and $g_{x_k}$ have} exactly $k$ zeros on $(0,+\infty)$;
\item 
$(f_{x_k},g_{x_k})$ converges exponentially to $(0,0)$ as $r\to+\infty$.
\end{enumerate}
\end{thm}

This theorem is the first result of existence of excited state solutions for the model studied in \cite{estebanrotanodarirad,estebanrotanodari} for which Esteban and Rota Nodari proved the existence of a ground state solution.

Our theorem is similar to the result obtained by Balabane, Cazenave,  Douady and Merle (\cite{balabane1988existence}) for a nonlinear Dirac equation. Our proof is based on a shooting method inspired by the one used by Balabane, Dolbeault and Ounaies (\cite{balabane2003}). 

In \cite{balabane1988existence}, the authors proved the existence of infinitely many stationary states for a nonlinear Dirac equation. More precisely, they showed the existence of a bounded  increasing sequence of positive initial data $\{x_k\}_k$ such that the associated solutions are global and each component has $k$ nodes. 

In \cite{balabane2003}, thanks to some estimations on the energy decay and the rotation speed, the authors proved the existence of infinitely many solutions for a sublinear elliptic equation.  As in \cite{balabane1988existence}, they showed the existence of an 
 increasing sequence of initial data $\{x_k\}_k$ such that the associated solutions are radial, compactly supported and have exactly $k$ nodes.



As we remarked above, the first difficulty to deal with here is that, to obtain a localized solution, the initial condition $x$ must be chosen in $(0,1)$. Moreover, we are looking for solutions such that each component has exactly $k$ zeros on $(0,+\infty)$.  

Usually in  a shooting method, the localized solution with $k$ nodes is obtained taking the solution whose initial data $x$ is  the supremum of a well-chosen open subset of $\{x:~g_x~\text{has $k$ zeros}\}$. Hence, the main difficulty of our shooting method is to prove that for any $k\in \mathbb{N}$, there exists $\epsilon>0$ such that
\[
	\{x\in(0,1):~g_x~\text{has $k$ zeros}\}\subset(0,1-\epsilon).
\]
To do this, we have to give some accurate estimations on the behavior of the solution when the initial condition $x$ becomes close to $1$. 
The presence of four rest points $(\pm\sqrt{a-b},\pm 1)$ in the Hamiltonian system
\begin{equation}\label{eq_autonome}
	\left\{\begin{aligned}
&f'=g(f^2-a g^2+b)\,\\
&g'=f(1-g^2)\,
\end{aligned}\right.,
\end{equation}
associated with the system \eqref{eq_nonautonome}, makes this study difficult. 
Indeed, we would like to control the solutions $(f_x,g_x)$ thanks to the continuity of the flow comparing $(f_x,g_x)$ to $(f_1,g_1)$ whenever $x$ is close enough to $1$. The problem is that $(f_1,g_1)$ tends to the rest point $(-\sqrt{a-b},1)$ of the system \eqref{eq_autonome}. Thus, $(f_x,g_x)$ stay in a neighborhood of $(-\sqrt{a-b},1)$ a very long time if $x$ is sufficiently close to $1$. Since $(f_1,g_1)$ does not wind around $(0,0)$, it is hopeless to get estimations on the speed of rotations of $(f_x,g_x)$ around $(0,0)$ as in \cite{balabane2003}. Hence, we introduce another strategy to prove that $(f_x,g_x)$ winds around $(0,0)$.

First of all, we prove that $(f_x,g_x)$ exits the neighborhoods of $(-\sqrt{a-b},1)$ at finite time, possibly very large. 
Next,  we want to control the position of $(f_x,g_x)$ when this occurs. To do this, we introduce the so-called Hamiltonian regularization. More precisely, we replace the system  \eqref{eq_nonautonome} by the Hamiltonian ones \eqref{eq_autonome} in a neighborhood of the points $(\pm\sqrt{a-b},\pm 1)$ (see Figure \ref{regsystem}). Then, we can use the  qualitative properties of the solutions of the Hamiltonian system \eqref{eq_autonome} to know the position of the solution when it exits the neighborhood of $(-\sqrt{a-b},1)$. Finally, we iterate the reasoning to prove that if $x$ is sufficiently close to $1$, then $g_x$ has more than $k$ zeros.

The idea of the Hamiltonian regularization is inspired by the proof of Le Treust in \cite{letreust2011}. In this paper, the author proved the existence of infinitely many compactly supported nodal solutions of a Dirac equation with singular nonlinearity. The main problem encountered is that the nonlinearity is singular and the main theorems of ODE fail to show local existence and uniqueness. To overcome this, Le Treust used a regularization by a Hamiltonian system whenever the problems occur. The advantage of such a regularization is that it gives a better control of the regularized solutions while keeping true some qualitative properties of the solutions of the non-autonomous system of equation.

In section \ref{secregularized}, we introduce the regularized system and we prove the existence of nodal localized solutions of the regularized problem assuming some key lemmas. In the next section, we prove these lemmas. In section \ref{secproofth}, we show that the localized nodal solutions of the original system \eqref{eq_nonautonome} can be obtained as limits of nodal localized solutions of the regularized system. Finally, in the appendix, we give some useful properties of the Hamiltonian energy associated to the system.

\section{The regularized problem and the shooting method}\label{secregularized}
\subsection{Construction of the regularized problem} \label{subsec_construtionreg}
\tikzstyle{background grid}=[draw, black!50,step=.5cm]
	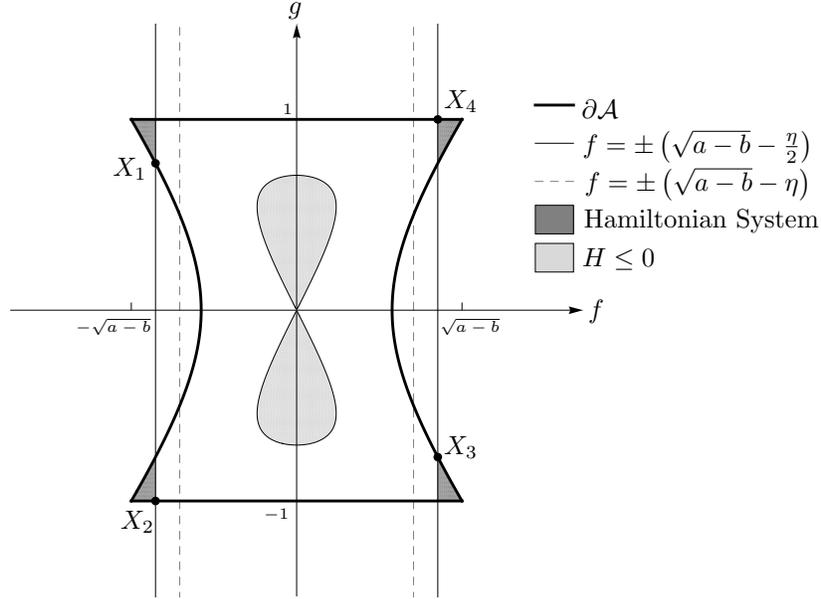
\begin{figure}[h!]
	\begin{center}
	\begin{tikzpicture}
         \node[inner sep=0pt,above right]{\includegraphics[]{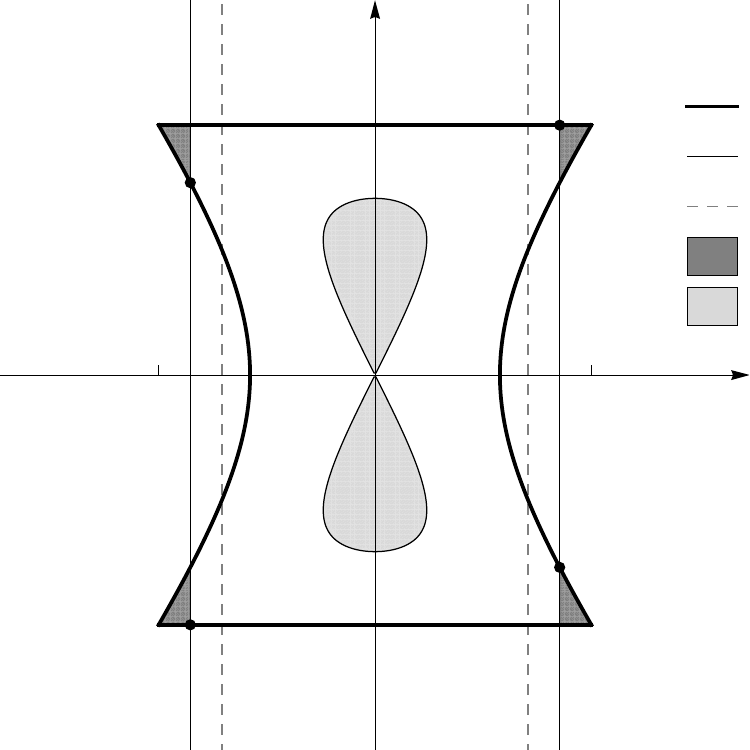}};
         \node (Hs) at (9.2,5) {Hamiltonian System};
         \node (H0) at (8.1,4.5) {$H\le 0$};
         \node (f) at (7.8,3.8) {$f$};
         \node (g) at (3.8,7.8) {$g$};
         \node (X1) at (1.6,5.7) {$X_1$};
         \node (X2) at (1.7,1) {$X_2$};
         \node (X3) at (6,2) {$X_3$};
         \node (X4) at (6,6.6) {$X_4$};
         \node (feta2) at (9.15,6) {$f=\pm\left(\sqrt{a-b}-\frac{\eta}{2}\right)$};
         \node (feta) at (9.15,5.5) {$f=\pm\left(\sqrt{a-b}-\eta\right)$};
         \node (A) at (7.85,6.5) {$\partial\mathcal A$};
         \node (ab) at (6.12,3.6) {\tiny $\sqrt{a-b}$};
         \node (-ab) at (1.38,3.6) {\tiny $-\sqrt{a-b}$};
         \node (1) at (3.7,6.5) {\tiny $1$};
         \node (-1) at (3.55,1.1) {\tiny $-1$};
	\end{tikzpicture}
	\caption{Regularized System}
	\label{regsystem}
	\end{center}
\end{figure}

Let
\[
	\varphi:(\eta,f,g)\in{\left(0,\sqrt{a-b}-\sqrt{\frac{a}{2}}\right)}\times[-\sqrt{a-b},\sqrt{a-b}]\times\mathbb{R}\mapsto\varphi_\eta(f,g)\in[0,1]
\]
be a smooth function on 
\[
	{\left(0,\sqrt{a-b}-\sqrt{\frac{a}{2}}\right)}\times(-\sqrt{a-b},\sqrt{a-b})\times\mathbb{R}
\]
 such that for all $(f,g)\in\mathbb{R}^2$, 
 all $\eta\in {\left(0,\sqrt{a-b}-\sqrt{\frac{a}{2}}\right)}$ 
\[
	\varphi_\eta(f,g) = \varphi_\eta(|f|,|g|)=\left\{
		\begin{array}{ll}
			0	&\text{if}~|f|\geq\sqrt{a-b}-\eta/2\\
			1	&\text{if}~|f|\leq\sqrt{a-b}-\eta.
		\end{array}
	\right.
\]

Consider the system of equations
\begin{equation}\label{eq_nonautonomereg}
	\left\{
		\begin{aligned}
			f'+\frac{2\varphi_\eta(f,g)}{r}f	&=g(f^2-ag^2+b),\\
			g'			&=f(1-g^2),
		\end{aligned}
	\right.
\end{equation}
and the Cauchy problem
\begin{equation}\label{pbeq_nonautonomereg}
	\left\{
		\begin{aligned}
			&f'+\frac{2\varphi_\eta(f,g)}{r}f	=g(f^2-ag^2+b),\\
			&g'			=f(1-g^2),\\
			&f(0)=0,~g(0)=x.
		\end{aligned}
	\right.
\end{equation}
We denote by $(f_{x,\eta},g_{x,\eta})$ the solutions to problem \eqref{pbeq_nonautonomereg}.
\begin{rem} When $\eta>0$,
in the neighborhood of the four points $(\pm\sqrt{a-b},\pm 1),$ the system of equations \eqref{eq_nonautonomereg} becomes the following autonomous one
\begin{equation}\label{eq_autonomebis}\tag{\ref{eq_autonome}}
	\left\{
		\begin{array}{rl}
			f'			&=g(f^2-ag^2+b),\\
			g'			&=f(1-g^2).
		\end{array}
	\right.
\end{equation}
This system is a Hamiltonian system associated with the energy
\begin{equation}\label{eqenergy}
	H(f,g)=\frac{1}{2}f^2(1-g^2)+\frac{a }{4}g^4-\frac{b}{2}g^2.
\end{equation}
\end{rem}
\begin{rem}
The behavior of the solutions  of \eqref{eq_autonome} is easier to understand than the one  of the solutions of  \eqref{eq_nonautonome}. This is actually the reason why we introduce such a Hamiltonian regularization in the neighborhood of the saddle points $(\pm\sqrt{a-b},\pm 1)$ of $H$. 
\end{rem}
\subsection{Properties of the regularized system}
We fix $\eta\in(0,\sqrt{a-b}-\sqrt{\frac{a}{2}})$.

We begin by studying the existence and the uniqueness of the solutions of (\ref{eq_nonautonomereg}).
\begin{lem}\label{lemexistencesoleta}
Let $x\in \mathbb{R}$. For any $a ,b>0$, there is $\tau_\eta>0$ and $(\fxeta,\gxeta)\in\mathcal{C}^1\left([0,\tau_\eta],\mathbb{R}^2\right)$ unique solution of (\ref{eq_nonautonomereg}) satisfying $\fxeta(0)=0$, $\gxeta(0)=x$. Moreover, $(\fxeta,\gxeta)$ can be extended on a maximal interval $[0,R_{x,\eta})$ with either $R_{x,\eta}=+\infty$ or $R_{x,\eta}<+\infty$ and $\lim_{r\to {R_{x,\eta}}}|\fxeta|+|\gxeta|=+\infty$. Furthermore, $(\fxeta,\gxeta)$ depends continuously on $x$ and $\eta$, uniformly on $[0,R]$ for any $R<R_{x,\eta}$.
\end{lem}
\begin{proof}
As in \cite{cazenavevazquez}, it is enough to write
\begin{align*}
f(r)&=\frac{1}{r^2}\int_{0}^rs^2g(s)(f^2(s)-a g^2(s)+b)+2s(1-\varphi_\eta(f(s),g(s)))f(s)\,ds\,,\\
g(r)&=x+\int_{0}^rf(s)(1-g^2(s))\,ds\,,
\end{align*}
and note that the right hand side of (\ref{eq_nonautonomereg}) is a Lipschitz continuous function of $(f,g)$. The lemma follows from a classical contraction mapping argument.
\end{proof}

Next, we define  $\mathcal{A}=\{(f_0,g_0)\in\mathbb{R}^2|\, 2f_0^2-ag_0^2-(a-2b)\le 0,\; g_0^2\le 1\}$, the  set of admissible points. 
\begin{rem*} If $g_0^2\le 1$, $(f_0,g_0)\in \mathcal{A}\,$ if and only if $\,H(f_0,g_0)\le H(0,1)=\frac{a-2b}{4}$ (see Figure \ref{regsystem}). \end{rem*}

A key property of the solutions of the system (\ref{eq_nonautonomereg}) is the behavior of the energy $H$ along the trajectories of the solutions  as stated in the following lemma. 
\begin{lem}\label{lem_deriveenergiereg}
	Let $x\in \mathbb{R}$ and $a ,b>0$. Then for any $r\in [0,R_{x,\eta})$ we have
	\[
		\frac{d}{dr}H(\fxeta,\gxeta)(r) = -\frac{2\varphi_\eta(\fxeta,\gxeta)}{r}\fxeta^2(1-\gxeta^2).
	\]
\end{lem}
\begin{rem}
Let us remark that this property which is true for the non-regularized system (\ref{eq_nonautonome}) is also true for the regularized one thanks to our choice of regularization.
\end{rem}
\begin{proof}
The proof is a straightforward calculation.
\end{proof}


Next, we prove a result that ensures that for all $x\in(0,1)$ the solutions $(\fxeta,\gxeta)$ are global and live in $\mathcal A$.

\begin{lem}\label{lempropertiessol} Let  $a ,b>0$ such that $a -2b>0$ and let $(\fxeta,\gxeta)$ be the solution of (\ref{eq_nonautonomereg}) satisfying $\fxeta(0)=0$, $\gxeta(0)=x$. If $x^2\leq1$, then $\gxeta^2(r)\leq1$ and $\fxeta^2(r)\leq a-b$ for all $r \in [0,R_{x,\eta})$ and $R_{x,\eta}=+\infty$. Moreover, $(\fxeta,\gxeta)(r)\in\mathcal{A}$, for all $r\in[0,+\infty)$  and if $x^2<1$, then $(\fxeta,\gxeta)(r)\in\mathring{\mathcal A}$ for all $r\in[0,+\infty)$.
\end{lem}


Lemma \ref{lempropertiessol} can be proved as in \cite{estebanrotanodarirad} using the monotonicity properties of the energy. For the reader's convenience, we rewrite the proof here. 

\begin{proof} 
First of all, we use the monotonicity of the function $F(x)=\frac{a }{4}x^4-\frac{b}{2}x^2$, the facts that $F(x)\le0$ in $ \left[-\sqrt{\frac{2b}{a }}, \sqrt{\frac{2b}{a }}\right]$ and $ \sqrt{\frac{2b}{a}}<1$ to show that $F(x)<F(1)$ for all $x$ such that $x^2<1$.

Let $\gxeta(0)=x$ such that $x^2<1$ and suppose, by contradiction, that there exists $r_0$ such that $\gxeta^2(r_0)=1$ and $\gxeta^2(r)<1$ for all $r\in [0,r_0)$.
As a consequence of Lemma \ref{lem_deriveenergiereg}, the energy $H(\fxeta,\gxeta)(r)$ is non-increasing on $[0,r_0)$, that means
$$
H(0,x)\ge H(\fxeta,\gxeta)(r_0) \; \mbox{, or equivalently, } \;
F(x) \ge F(1).
$$
The above inequality contradicts the properties of $F$. As a conclusion, $\gxeta^2(r)<1$ for all $r\in [0,R_{x,\eta})$.

Then, applying Lemma \ref{lem_deriveenergiereg}, we obtain that the energy is non-increasing. Thus, 
$$
H(\fxeta,\gxeta)(r)\le H(0,x)<\frac{a-2b}{4}\,,\;\forall r\in[0,R_{x,\eta})\,,$$
and by the remark following the definition of the set $\mathcal{A}$, $(\fxeta,\gxeta)(r)\in\mathring{\mathcal{A}}$ and $\fxeta^2(r)<a-b$, for all $r\in[0,R_{x,\eta})$. In particular, $R_{x,\eta}=+\infty$. The case $x=\pm1$ is straightforward.
\end{proof}

\begin{rem}\label{rem_decroissanceenergie} As a consequence of Lemma \ref{lem_deriveenergiereg} and Lemma \ref{lempropertiessol}, if $x^2\leq 1$, the energy $H(\fxeta,\gxeta)(r)$ is non-increasing on $[0,+\infty)$.
\end{rem}


{ Then, we state the following perturbation result.}
\begin{lem}\label{lemconvconssol} 
Let $(\tilde f,\tilde g)\in \mathcal{A}$. Let $(f,g)$ be the solution of (\ref{eq_autonome}) with initial data $(\tilde f,\tilde g)$. Let $(\tilde f^n,\tilde g^n)\in\mathcal{A}$ and $\rho_n$ be such that 
\begin{align*}
\lim\limits_{n\to+\infty}\rho_n=+\infty&\mbox{ and }\lim\limits_{n\to+\infty}(\tilde f^n,\tilde g^n)=(\tilde f, \tilde g).
\end{align*}
Let $(f^n_\eta,g^n_\eta)$ be a solution of 
\begin{equation*}
\left\{\begin{aligned}
(f^n)'+\frac{2\varphi_\eta(f^n,g^n)}{\rho_n+r}f^n&=g^n((f^n)^2-a (g^n)^2+b)\\
(g^n)'&=f^n(1-(g^n)^2)
\end{aligned}\right.
\end{equation*}
such that $f^n(0)=\tilde f^n$, $g^n(0)=\tilde g^n$. 
Then
\begin{enumerate}
\item $(f^n_\eta,g^n_\eta)$ is defined on $[0,+\infty),$
\item $(f^n_\eta,g^n_\eta)$ converges to $(f,g)$ uniformly on bounded intervals.
\end{enumerate} 
\end{lem}

The proof of this lemma is a straightforward modification of the proof of \cite[Lemma 3.2]{estebanrotanodarirad}.

Next, we introduce the following definition to count the number of times  that the solutions cross the set $\{g=0\}.$
\begin{defn}\label{def_signe}
We say that a continuous function $g$ defined on an interval $I$ changes sign at $r_0\in I$ if $g(r_0)=0$ and there exists $\epsilon>0$ such that $[r_0-\epsilon,r_0+\epsilon]\subset I,$
\[
	g(r_0-\epsilon)g(r_0+\epsilon)<0~\text{and}~g(r_0-r)g(r_0+r)\leq0
\]
for all $r\in(0,\epsilon).$
\end{defn}

\begin{rem}\label{remequivsignvanish} Let $(\fxeta,\gxeta)$ be a nontrivial solution of \eqref{pbeq_nonautonomereg} with $x\in(0,1)$. Then, $\gxeta$ changes sign at $0\le r_0<+\infty$ if and only if $\gxeta$ vanishes at $0\le r_0<+\infty$. 

Indeed, since $(\fxeta,\gxeta)$ is a nontrivial solution and $r_0<+\infty$, $\fxeta(r_0)\neq 0$. Hence, $\gxeta'(r_0)=\fxeta(r_0)(1-\gxeta(r_0)^2)\neq 0$ and $\gxeta$ changes sign at $r_0$.
\end{rem}

Finally, we state the following lemma which gives us an important qualitative property of the solutions of the system (\ref{eq_nonautonomereg}).

\begin{lem}\label{lem_tourexpdecay} Let $x\in(0,1)$. If $(\fxeta,\gxeta)$ is a solution of \eqref{pbeq_nonautonomereg} such that $\gxeta$ changes sign a finite number of times and 
$$
\lim_{r\to+\infty}H(\fxeta,\gxeta)(r)\ge 0,
$$
then, for all $r\ge 0$, 
\begin{equation}\label{eqexpdecay}
\vb \fxeta(r)\vb+\vb \gxeta(r)\vb\le C\exp(-K_{a,b}r)
\end{equation}
with $K_{a,b}=\min\left\{\frac{b}{2},\frac{2a-b}{2a}\right\}$ and $C$ a positive constant.

In particular, we get
\begin{equation}\label{eqconvsol}
\lim_{r\to+\infty}(\fxeta,\gxeta)(r)=(0,0).
\end{equation}
\end{lem}

\begin{proof} 
 First of all, we remark that if $\lim\limits_{r\to+\infty}\gxeta(r)=\delta$, then $|\delta|\ne1$. Moreover, if $\delta\ne 0$ then $(\fxeta,\gxeta)$ tends either to $(0,\sqrt{\frac{b}{a}})$ or to $(0,-\sqrt{\frac{b}{a}})$ as $r$ goes to $+\infty$. Indeed, suppose by contradiction that $\delta=\pm1$, then 
\[
	\lim\limits_{r\rightarrow+\infty}H(\fxeta,\gxeta)(r) = H(0,1),
\]
which contradicts the monotonicity of $H$ (Lemma \ref{lem_deriveenergiereg}) since $H(0,x)<H(0,1)$ and $x\in (0,1)$.
Hence, $-1< \delta<1$. Next, suppose $\delta\neq 0$ and let $\{r_n\}_n$ be a sequence such that $\lim\limits_{n\to +\infty}r_n=+\infty$ and $\lim\limits_{n\to+\infty}\fxeta(r_n)=k$ for some $k\in \mathbb{R}$. Let $(u,v)$ be the solution of \eqref{eq_autonome} with initial data $(k,\delta)$. It follows from Lemma \ref{lemconvconssol} that $(\fxeta(r_n+\cdot),\gxeta(r_n+\cdot))$ converges uniformly to $(u,v)$ on bounded intervals. Since, $\lim\limits_{n\to+\infty}\gxeta(r_n+r)=\delta$ for any $r>0$, we have $v(r)=\delta\in(-1,0)\cup(0,1)$ for any $r\ge 0$. Hence, from the second equation of \eqref{eq_autonome},  we obtain $u(r)=0$ for all $r>0$. This means that $(u,v)$ is an equilibrium point of \eqref{eq_autonome}, and, since $\delta\in(-1,0)\cup(0,1)$, this implies $k=0$ and $\delta=\pm\sqrt\frac{b}{a}$. As a conclusion, $\fxeta$ converges as $r$ goes to $+\infty$, $\lim\limits_{r\to+\infty}(\fxeta,\gxeta)(r)=\left(0,\pm\sqrt\frac{b}{a}\right)$ and 
$$
\lim\limits_{r\to+\infty}H(\fxeta,\gxeta)(r)= H\left(0,\sqrt\frac{b}{a}\right)<0.
$$

Next, we claim that, if $\gxeta$ changes sign a finite number of times and 
$$\lim\limits_{r\to+\infty}H(\fxeta,\gxeta)(r)\ge 0,$$ then there exists $\bar R<+\infty$ such that 
\begin{itemize}
\item or $\gxeta(r)>0$ and $\fxeta(r)<0$ for all $r>\bar R$, 
\item or $\gxeta(r)<0$ and $\fxeta(r)>0$ for all $r>\bar R$.
\end{itemize}
Indeed, by Remark \ref{remequivsignvanish}, if $\gxeta$ changes sign a finite number of times, then $\gxeta$ vanishes a finite number of times and there exists $R<+\infty$ such that $\gxeta(r)>0$ or $\gxeta(r)<0$ for all $r>R$. Thanks to the symmetries of the problem, we can suppose w.l.o.g. that $\gxeta(r)>0$ for all $r>R$.

Hence, it remains to prove that there exists $R< \bar R<+\infty$ such that $\fxeta(r)<0$ for all $r>\bar R$. We proceed as follows : first we prove that we cannot have $\fxeta(r)>0$ for all $r> R$ and second we show that $\fxeta$ vanishes at most once in $[R,+\infty)$.

\noindent \emph{Step 1.} Suppose, by contradiction, that $\fxeta(r)>0$ for all $r>R$. This implies that $\gxeta(r)$ is increasing for all $r>R$ and $\lim\limits_{r\to+\infty}\gxeta(r)=\delta$ with $0<\delta\leq1$. 

 Hence, as we proved above, we have $\lim\limits_{r\to+\infty}(\fxeta,\gxeta)(r)=\left(0,\sqrt\frac{b}{a}\right)$ and
$$
\lim\limits_{r\to+\infty}H(\fxeta,\gxeta)(r)= H\left(0,\sqrt\frac{b}{a}\right)<0
$$ 
which contradicts the fact that $\lim\limits_{r\to+\infty}H(\fxeta,\gxeta)(r)\ge 0$ (see Remark \ref{remzerosf} for an alternative proof).

As a consequence, there exists $R<\bar R<+\infty$ such that $\fxeta(\bar R)=0$. Let us remark moreover that for such a $\bar R,$ we have $\fxeta'(\bar R)<0.$

\noindent \emph{Step 2.} Suppose next, by contradiction, that there exist { a positive constant $R'$} such that $R<\bar R<R'<+\infty,$  $\fxeta(R')=0$ and $\fxeta(r)<0$ on $(\bar R,R')$. Since $\fxeta'$ has to be nonnegative in a neighborhood of $R'$, we can conclude that $0<\gxeta(R')\le \sqrt\frac{b}{a}$. Hence, 
$$
\lim\limits_{r\to+\infty}H(\fxeta,\gxeta)(r)\le H\left(0,\gxeta(R')\right)<0
$$ 
which contradicts the fact that $\lim\limits_{r\to+\infty}H(\fxeta,\gxeta)(r)\ge 0$.

As a conclusion, there exists $\bar R<+\infty$ such that $\gxeta(r)>0$ and $\fxeta(r)<0$ for all $r>\bar R$. This implies that $\gxeta(r)$ is decreasing for all $r>\bar R$ and $\lim\limits_{r\to+\infty}\gxeta(r)=\delta$ with $0\le\delta<1$.  We claim that $\delta=0$. Indeed, suppose by contradiction, $\delta\neq 0$. As above, we obtain $\lim\limits_{r\to+\infty}(\fxeta,\gxeta)(r)=\left(0,\sqrt\frac{b}{a}\right)$ and
$$
\lim\limits_{r\to+\infty}H(\fxeta,\gxeta)(r)= H\left(0,\sqrt\frac{b}{a}\right)<0
$$ 
which contradicts the fact that $\lim\limits_{r\to+\infty}H(\fxeta,\gxeta)(r)\ge 0$ (see Remark \ref{remzerosf} for an alternative proof).

Hence, $\lim\limits_{r\to+\infty}\gxeta(r)=0$ and $\gxeta^2(r)\le \frac{b}{2a}$ for $r$ large enough. Considering \eqref{eq_nonautonomereg},
$$
\fxeta'\ge -\frac{2\varphi_\eta(\fxeta,\gxeta)}{r}\fxeta+\frac{b}{2}\gxeta\,,\quad \gxeta'\le \frac{2a-b}{2a}\fxeta.
$$
Thus, for $r$ large enough, 
$$
(\gxeta-\fxeta)'+K_{a,b}(\gxeta-\fxeta)\le 0,
$$
with $K_{a,b}=\min\left\{\frac{b}{2},\frac{2a-b}{2a}\right\}$. Integrating the above equation, we obtain
$$
|\fxeta(r)|+ |\gxeta(r)|\le C\exp(-K_{a,b}r)
$$
for all $r\geq 0$ with $C>0$.

With exactly the same arguments, we treat the case $\gxeta(r)<0$ for all $r>R$.
\end{proof}

{ \begin{rem}\label{remzerosf} The proof is very similar to the one of \cite[Lemma 3.4]{estebanrotanodarirad}. With the same arguments of \cite[Proof of Lemma 3.4]{estebanrotanodarirad}, we can prove that 
if $x\in(0,1)$ and $(\fxeta,\gxeta)$ is a solution of \eqref{pbeq_nonautonomereg} such that $$\lim\limits_{r\to+\infty}(\fxeta,\gxeta)(r)=\left(0,\pm\sqrt\frac{b}{a}\right),$$ then $\fxeta$ has infinitely many zeros.

This property is equivalent to the fact that $(\fxeta,\gxeta)$ cannot tend to $\left(0,\pm\sqrt\frac{b}{a}\right)$, while being in one of the half-planes $\{f>0\}$ or $\{f<0\}$.

This remark allows us to prove in an alternative way that if $\gxeta(r)>0$ for all $r>R$, then $\fxeta$ has to vanish at least once in $(R,+\infty)$ without using the fact that  $\lim\limits_{r\to+\infty}H(\fxeta,\gxeta)(r)\ge 0$ (\emph{Step 1} of the proof of Lemma \ref{lem_tourexpdecay}).

Moreover, it proves also that if $\gxeta(r)>0$  and $\fxeta(r)<0$ for all $r>\bar R$, then $\lim\limits_{r\to+\infty}(\fxeta,\gxeta)(r)=(0,0)$ without using $\lim\limits_{r\to+\infty}H(\fxeta,\gxeta)(r)\ge 0$  (end of the proof of Lemma \ref{lem_tourexpdecay}).
\end{rem}}


%
%
%
%

\subsection{The shooting method}
\tikzstyle{background grid}=[draw, black!50,step=.5cm]
	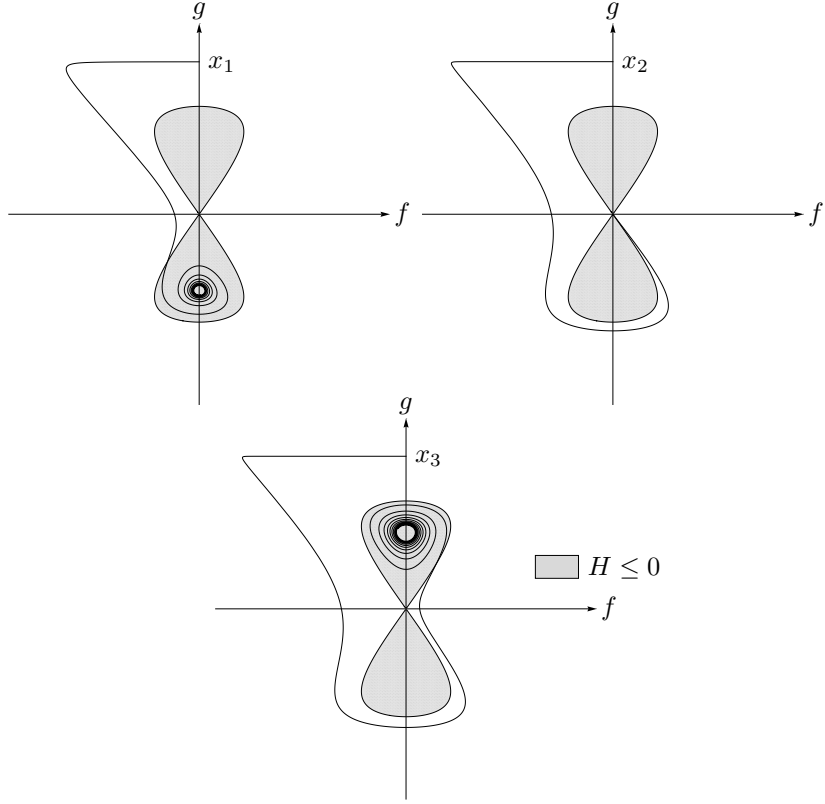
\begin{figure}[h!]
	\begin{center}
	\begin{tikzpicture}
         \node () at (0,0) {\includegraphics[]{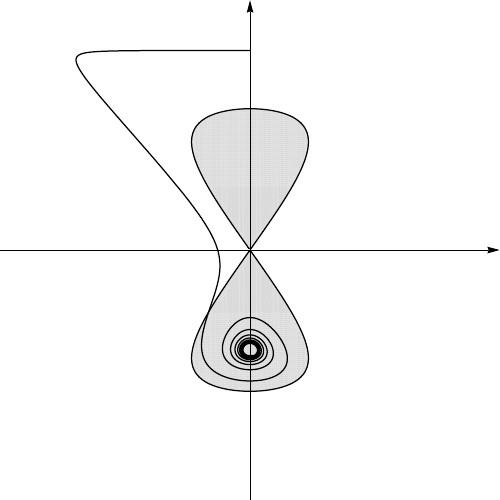}};
         \node (f1) at (2.7,0) {$f$};
         \node (g1) at (0,2.7) {$g$};
         \node (x1) at (0.3,2) {$x_1$};
         \node () at (5.5,0) {\includegraphics[]{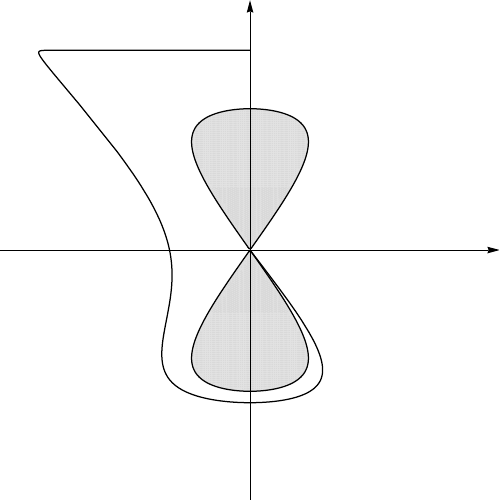}};
         \node (f2) at (2.7+5.5,0) {$f$};
         \node (g2) at (5.5,2.7) {$g$};
         \node (x2) at (5.8,2) {$x_2$};
         \node () at (2.75,-5.25) {\includegraphics[]{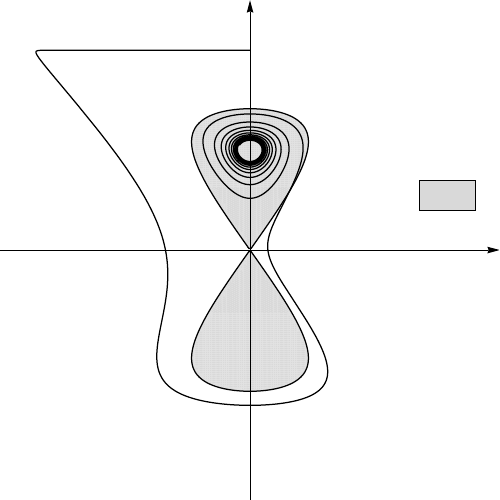}};
         \node (f3) at (2.7+2.75,-5.25) {$f$};
         \node (g3) at (2.75,2.7-5.25) {$g$};
         \node (x3) at (0.3+2.75,2-5.25) {$x_3$};
         \node (H0) at (2.9+2.75,0.55-5.25) {$H\le 0$};
	\end{tikzpicture}
	\caption{$x_1\in A_1$, $x_2\in I_1$, $x_3\in A2$}
	\label{figAkIk}
	\end{center}
\end{figure}
Following \cite{balabane2003}, we define $I_{-1}=\emptyset$ and, for $k\in\mathbb{N}$ and $\eta\in(0,\sqrt{a-b}-\sqrt{\frac{a}{2}})$ fixed,
\begin{align*}
	&A_k 	= \{x\in(0,1):\underset{r\rightarrow+\infty}{\lim} H(\fxeta,\gxeta)(r)<0, \gxeta\;  \text{changes sign $k$ times on $\mathbb{R}^+$}\},\\
	&I_k 	= \{x\in(0,1):\underset{r\rightarrow+\infty}{\lim} (\fxeta,\gxeta)(r)=(0,0), \gxeta\;  \text{changes sign $k$ times on $\mathbb{R}^+$}\}
\end{align*}
(see Figure \ref{figAkIk}).

\begin{rem}
By Lemma \ref{lempropertiessol}, we get that $(\fxeta,\gxeta)(r)\in \mathcal{A}$ for all $r$ whenever $x\in[0,1]$. Remark \ref{rem_decroissanceenergie} ensures then that $\underset{r\rightarrow+\infty}{\lim}~ H(\fxeta,\gxeta)(r)$ exists for all $x\in [0,1]$.
\end{rem}
\begin{rem}
We want to find non trivial localized solutions of equations \eqref{eq_nonautonomereg} with a given number of nodes that is to say, $x\in (0,1)$ such that 
\[
	\underset{r\rightarrow +\infty}{\lim}~(\fxeta,\gxeta)(r)=(0,0)
\] 
and $\gxeta$ changes sign $k$ times on $\mathbb{R}^+$. 
To do this, we show by a shooting method that 
\[
	I_k\ne \emptyset
\]
for all $k\in \mathbb{N}$ and all $\eta\in(0,\sqrt{a-b}-\sqrt{\frac{a}{2}}).$
\end{rem}
The core of the shooting method is the following lemma which gives the main properties of the sets $A_k$ and $I_k.$ It is very similar to the properties stated in the proof of \cite[Theorem 1]{balabane2003} except that the sets $A_k$ and $I_k$ are always bounded since they are included in $(0,1)$. The good equivalent property which is adapted to our case is given by point \eqref{lem_AkIkborne} of the next lemma.
\begin{lem}\label{lem_proprieteAkIk}
	For all $k$ in $\mathbb{N}$ and all $\eta\in(0,\sqrt{a-b}-\sqrt{\frac{a}{2}})$ we have
	\begin{enumerate}[(i)]
		\item \label{lem_Akouvert} 	$A_k$ is an open set,
		\item \label{lem_AkIkborne} 	there is $\epsilon\in (0,1)$ such that $A_k\cup I_k\subset (0,1-\epsilon),$
				\item \label{lem_trapping} 	if $x\in I_k, $ there exists  $\epsilon>0$ such that $(x-\epsilon,x+\epsilon)\subset A_k\cup I_k \cup A_{k+1}$
		\item \label{lem_supAk} 	if $A_k$ is not empty, we have $\sup A_k\in I_{k-1}\cup I_k,$
		\item \label{lem_supIk} 	if $I_k$ is not empty, we have $\sup I_k\in  I_k,$
	\end{enumerate}
\end{lem}

The proof of this lemma is given in Section \ref{sec_proofsshootingproperties}. We are now able to prove the following proposition.
\begin{prop}\label{prop_shooting}
	There exists an increasing sequence $\{x_k\}_{k\geq0}\subset (0,1)$ such that $x_k\in I_k$ for all $k\in \mathbb{N}$.
\end{prop}  
The proof is essentially the same as in \cite{balabane2003}. We write it down here for sake of completeness. 
\begin{proof}
 	We prove by induction that for all $k\in \mathbb{N},$  
	\begin{align*}
		&A_k\ne\emptyset,\\
		&\sup I_{k-1}< \sup A_k.
	\end{align*}
	If this property is true for all $k$, then $A_k$ is not empty, $\sup A_k\in I_k$ by point \eqref{lem_supAk} of Lemma \ref{lem_proprieteAkIk} and $\sup A_{k}\leq \sup I_{k}<\sup A_{k+1}$. Hence, if we choose $x_k=\sup A_k$ we get the proposition.
	\begin{enumerate}
		\item Let $k=0.$ We have that for all $x\in(0,\sqrt{\frac{2b}{a}})$ 
	\[
		H(0,x)<0.
	\]
	Thus, Lemma \ref{lempropertiessol} and remark \ref{rem_decroissanceenergie} ensure that 
	\[
		(0,\sqrt{\frac{2b}{a}}) \subset A_0
	\]
	and
	\[
		-\infty=\sup I_{-1}<\sup A_0.
	\]
	\item Let us assume now that for some $k\in\mathbb{N}$, we have
	\begin{align*}
		&A_k\ne\emptyset,\\
		&\sup I_{k-1}< \sup A_k.
	\end{align*}
	By point \eqref{lem_supAk} of Lemma \ref{lem_proprieteAkIk}, we get $\sup A_k\in I_{k}$ which implies $I_k\ne \emptyset$ and { $\sup A_k\leq \sup I_k$}. Since $I_k\ne \emptyset$, by point \eqref{lem_supIk}, we obtain that $\sup I_k\in I_k$ and, since $\sup A_k\leq \sup I_k$, point \eqref{lem_trapping} ensures that there is $\epsilon>0$ such that 
	\[
		(\sup I_k,\sup I_k +\epsilon) \subset A_{k+1}.
	\]
	As a conclusion, we have
		\begin{align*}
		&A_{k+1}\ne\emptyset,\\
		&\sup I_{k}< \sup A_{k+1}.
	\end{align*}
	\end{enumerate}
\end{proof}
%
%
%
\section{Proof of Lemma \ref{lem_proprieteAkIk}}\label{sec_proofsshootingproperties}
In this section, we fix $\eta\in (0,\sqrt{a-b}-\sqrt{\frac{a}{2}})$ and we prove Lemma \ref{lem_proprieteAkIk}. 

%
\subsection{Proof of point \eqref{lem_AkIkborne} of Lemma \ref{lem_proprieteAkIk}}
\begin{rem}
This proof is the most technical point of the paper and contains the main novelties of our work. The introduction of the Hamiltonian regularization of subsection \ref{subsec_construtionreg} allow us 
 to control the behavior of the solution in the neighborhood of the stationary points $(\pm\sqrt{a-b},\pm 1)$ of the autonomous system of equations \eqref{eq_autonome}.
\end{rem}
We show by induction that for all $k$, there is  $\epsilon\in(0,1)$  such that if $x\in(1-\epsilon,1)$ then $\gxeta$ has at least $k+1$ changes of sign on $\mathbb{R}^+.$ This implies 
\[
	\underset{0\leq i\leq k}{\bigcup}A_i\cup I_i\subset (0,1-\epsilon)
\]
and  point \eqref{lem_AkIkborne} follows. 
\begin{rem}
The idea of the proof is that we can control the solutions $(\fxeta,\gxeta)$ thanks to the continuity of the flow on the parameter $x$ (see Lemma \ref{lemexistencesoleta}) comparing $(\fxeta,\gxeta)$ to $(f_{1,\eta},g_{1,\eta})$ on an interval of the type $[0,R]$ for $R>0$. Moreover, $(f_{1,\eta},g_{1,\eta})$ tends to a stationary point $(-\sqrt{a-b},1)$ of the system \eqref{eq_autonome}. Thus, $(\fxeta,\gxeta)$ stay in a neighborhood of $(-\sqrt{a-b},1)$ a very long time if $x$ is { sufficiently} close to $1$. We also know thanks to Lemma \ref{lem_tourexpdecay} that $(\fxeta,\gxeta)$ exits this neighborhood at finite time, possibly very large. The problem is that we have to control the position of $(\fxeta,\gxeta)$ when this occurs. To do this, we replace the system  \eqref{eq_nonautonome} by the Hamiltonian ones \eqref{eq_autonome} in this neighborhood. Then, we can use the  conservation of the energy $H$ along the trajectory of $(\fxeta,\gxeta)$ to know the position of $(\fxeta,\gxeta)$ when it exits the neighborhood of $(-\sqrt{a-b},1)$. 

After that, we can control the solutions $(\fxeta,\gxeta)$ thanks to the continuity of the flow comparing $(\fxeta,\gxeta)$ to a solution $(f,g)$ of \eqref{eq_autonome} that remains at all times on $\partial \mathcal{A}$ and tends to $(-\sqrt{a-b},-1)$ at infinity. We get that if $x$ is close enough to $1$ then $\gxeta$ changes sign one time. We iterate this reasoning {to obtain a solution for which $\gxeta$ changes sign more than $k$ times on $\mathbb R^+$}.
\end{rem}

{
\textbf{Step 1. Proof by induction}

First of all, we take $f_0 = \sqrt{a-b}-\eta/2>\sqrt{\frac{a}{2}}>\sqrt{\frac{a-2b}{2}}$ and we define 
\begin{align*}
X_1&:=\left(-f_0,\sqrt{\frac{2f^2_0-(a-2b)}{a}}\right)\\
X_2&:=\left(-f_0,-1\right)\\
X_3&:=\left(f_0,-\sqrt{\frac{2f^2_0-(a-2b)}{a}}\right)\\
X_4&:=\left(f_0,1\right).
\end{align*}
The points $X_i$ are on $\partial \mathcal A$,  for $i=1,\ldots,4$.
Furthermore, remind that $\varphi_\eta(f,g)=0$ whenever $\vb f\vb \ge f_0$ (see Figure \ref{regsystem}).

\begin{defn}
Let $k\in\mathbb{N}$ and $i\in\{1,\dots,4\}$ be given. We denote by $(H^k_i)$ the following property:
\begin{enumerate}
\item[]
for all $\gamma$ and $R$ positive constants given, there exists $\epsilon>0$ such that for any $x\in(1-\epsilon,1)$, there exists a positive constant $\tilde{R}>R$ which satisfies
\[
	(\fxeta,\gxeta)(\tilde{R})\in B(X_i, \gamma)\cap \mathcal{A}
\]
and such that $\gxeta$ change $k$ times of sign in $[0,\tilde{R}]$.
\end{enumerate}
\end{defn}

In the second step, we show that the properties $(H^0_1)$ is true. Next, in the third step, we prove that for $k\in\mathbb{N}$ given we have
\[\left\{\begin{array}{l}
	(H^k_1)\Rightarrow (H^{k+1}_2),\\
	(H^k_2)\Rightarrow (H^{k}_3),\\
	(H^k_3)\Rightarrow (H^{k+1}_4),\\
	(H^k_4)\Rightarrow (H^{k}_1)
\end{array}\right.\]
so that 
\[
	(H^k_1)\Rightarrow (H^{k+2}_1).
\]
As a consequence, we get by induction that the property $(H^{2k}_1)$ is true for all $k\in\mathbb{N}.$ In particular, there is $\epsilon\in(0,1)$ such that for all $x\in(1-\epsilon,1),$ $\gxeta$ changes at least $2k$ of sign on $[0,+\infty)$ so that 
\[
	I_i\cup A_i\subset (0,1-\epsilon).
\]
for all $i\in\{0,1,\dots,2k-1\}$ and point \eqref{lem_AkIkborne} of Lemma \ref{lem_proprieteAkIk} is proved.

\textbf{Step 2. Initialization:} 
We prove that $(H^0_1)$ is true.
%
%
%
%
\begin{enumerate}
\item
\textbf{Preliminary results.} 
Let $\gamma$ and $R$ be positive constants given.
First of all, remark that with the notation of Lemma \ref{lem_EtudeReg}, $$X_1=(-f_0,G_2(H(0,1))).$$  
 By continuity of $G_2,$ there exists $\delta>0$ such that $H(0,1)-\delta>E_c$ and 
 \[
 	\|(-f_0,G_2(E))-X_1\|<\gamma
 \]
 for all $E\in (H(0,1)-\delta,H(0,1))$ where $\|.\|$ is the Euclidean norm of $\mathbb{R}^2$.
So, we have to prove that  there exists $\epsilon>0$ such that for any $x\in(1-\epsilon,1)$, there exists a positive constant $R_1>R$ which satisfies
\[
H(0,1)-\delta<H(\fxeta,\gxeta)(R_1)< H(0,1),
\]
 $\fxeta(R_1)=-f_0$, $\gxeta(R_1)=G_2(H(\fxeta,\gxeta)(R_1))$ and $\gxeta$ does not change sign in $[0,R_1]$.

\item
\textbf{Control of the solutions of \eqref{pbeq_nonautonomereg} in an interval $[0,\overline R]$ with   $\overline R>0$.}
We denote $(f,g)$ the solution of Cauchy problem \eqref{pbeq_nonautonomereg} with $x=1$. It is easy to see that 
\[
	H(f,g)(r)=H(0,1),~g(r)=1,~f(r)>-\sqrt{a-b}~\text{for all}~r\in[0,+\infty)
\]
and
\[
	\underset{r\rightarrow+\infty}{\lim}~(f,g)(r)=(-\sqrt{a-b},1).
\]
As a consequence, there exists $\overline{R}>R$ such that  for all $r\geq \overline{R}$ 
\[
	(f,g)(r)\in \mathcal{A}\cap\{(f,g):~|f|>f_0\}.
\]
Next, since $H$ is continuous on $\mathbb{R}^2,$ there exists $0<\delta'<1$ such that for any 
\[
	(\overline f,\overline g)\in B((f,g)(\overline{R}),\delta'),
\] 
we have 
\[
	|H(\overline f,\overline g)-H(0,1)|<\delta
\]
where $B((f,g)(\overline{R}),\delta')$ is the Euclidean ball of $\mathbb{R}^2$ centered in $(f,g)(\overline{R})$ of radius $\delta'$. Moreover, if we choose $\delta'$ sufficiently small, we can assume that $|u|>f_0$ for all
\[
	(u,v)\in B((f,g)(\overline{R}),\delta').
\]
Finally, by Lemma \ref{lemexistencesoleta},  there exists $\epsilon\in(0,1)$ such that for all $x\in(1-\epsilon,1)$, 
\[
	\|(\fxeta,\gxeta)-(f,g)\|_{\infty,[0,\overline{R}]}\leq \delta'
\]
where $\|.\|_{\infty,[0,\overline{R}]}$ is the uniform norm of $\mathcal{C}([0,\overline{R}],\mathbb{R}^2)$ so that $\gxeta$ is positive in $[0,\overline R].$ 
\item
\textbf{Control of the solutions of \eqref{pbeq_nonautonomereg} in} $\mathcal{A}\cap\{(f,g)\in \mathbb R^2:~|f|>f_0\}.$
We define
\[
	R_1:=\inf\{r>\bar R: \vb\fxeta(r)\vb \leq f_0\}.
\]
By Lemma \ref{lemposencorner}, we have that
\[
	H(\fxeta,\gxeta)(r)>0,
\]
for all $r\in[\overline{R},R_1)$.
Moreover by Lemma \ref{lem_EtudeReg}, since $(\fxeta,\gxeta)(r)\in \mathcal A \cap \{(f,g)\in \mathbb R^2:~|f|>f_0\}$,
\[
	\gxeta(r)\geq \sqrt{\frac{2}{a}(f_0^2+b)-1}
\]
for all $r\in[\overline{R},R_1)$. Hence, by Lemma \ref{lem_tourexpdecay}, we get that $R_1$ is well-defined, $\bar R<R_1<+\infty$ and $\gxeta$ does not change sign in $[0,R_1]$.
Furthermore, $\fxeta(R_1)=-f_0$ and, since $(\fxeta,\gxeta)$ is solution of the Hamiltonian system of equation \eqref{eq_autonome} on $[\overline{R},R_1]$, we obtain
$$
H(\fxeta,\gxeta)(R_1)=H(\fxeta,\gxeta)(\overline{R})\in (H(0,1)-\delta,H(0,1)).
$$
Hence, it remains to show that $\gxeta(R_1)=G_2(H(\fxeta,\gxeta)(R_1))$. 


Let
\[
	\tilde R:=\inf\{r>0:~|\fxeta(r)|>f_0\},
\]
then $\fxeta(\tilde R)=-f_0=\fxeta(R_1)$. Moreover, in $[\tilde R,R_1]$, $(\fxeta,\gxeta)$ is solution of the Hamiltonian system \eqref{eq_autonome}; this implies  
$$
H(\fxeta,\gxeta)(\tilde R)=H(\fxeta,\gxeta)(R_1).
$$
Finally, $\gxeta$ is decreasing  on $[\tilde R,R_1]$; in particular
\[
	\gxeta(R_1)<\gxeta(\tilde R).
\]
Hence, by Lemma \ref{lem_EtudeReg}, we deduce 
\begin{align*}
\gxeta(\tilde R) = G_1(H(\fxeta,\gxeta)(R_1)),\\
\gxeta(R_1) = G_2(H(\fxeta,\gxeta)(R_1)).
\end{align*}
Thanks to the remark we did in the preliminary results, we proved Step 2.
\end{enumerate}
}

\textbf{Step 3. Iteration: } Let $k\in \mathbb N$ and suppose that property $(H^k_1)$ is true. We show that this implies  property $(H^{k+1}_2)$.
The proof of this fact is similar to the one of Step 2 except that now $(f,g)$ is a solution of autonomous system \eqref{eq_autonome}. 
{ \begin{enumerate}
\item
\textbf{Preliminary results.} 
Let $\gamma$ and $R$ be positive constants given.
First of all, remark that with the notation of Lemma \ref{lem_EtudeReg}, $$X_2=(-f_0,-G_1(H(0,1))).$$  
 By continuity of $G_1,$ there exists $\delta>0$ such that $H(0,1)-\delta>E_c$ and 
 \[
 	\|(-f_0,-G_1(E))-X_2\|<\gamma
 \]
 for all $E\in (H(0,1)-\delta,H(0,1))$.
So, we have to prove that  there exists $\epsilon>0$ such that for any $x\in(1-\epsilon,1)$, there exists a positive constant $R'>R$ which satisfies
\[
H(0,1)-\delta<H(\fxeta,\gxeta)(R')< H(0,1),
\]
$\fxeta(R') = -f_0$, $\gxeta(R')=-G_1(H(\fxeta,\gxeta)(R'))$ and $\gxeta$ changes sign $k+1$ times in $[0,R']$.

\item
\textbf{Control of the solutions of \eqref{pbeq_nonautonomereg} when the solutions exit a neighborhood of $X_1$.}
We denote by $(f,g)$ the solution of  the following autonomous system
\[
\left\{
	\begin{array}{ll}
		\eqref{eq_autonome}\\
		(f,g)(0)=X_1.
	\end{array}
\right.
\]
It is clear that 
\begin{align*}
	&H(f,g)(r)=H(0,1),\\
	&-1<g(r)<1,~f(r)>-\sqrt{a-b}~\text{for all}~r\in[0,+\infty),
\end{align*}
and
\[
	\underset{r\rightarrow+\infty}{\lim}~(f,g)(r)=(-\sqrt{a-b},-1).
\]
Hence, there is $\overline{R}>0$ such that  for all $r\geq \overline{R}$ 
\[
	(f,g)(r)\in \mathcal{A}\cap\{(f,g):~|f|>f_0\}.
\]
Next, since $H$ is continuous on $\mathbb{R}^2$, there exists $\delta'>0$ such that for any 
\[
	(\tilde f,\tilde g)\in B((f,g)(\overline{R}),\delta'),
\] 
we have 
\[
	|H(\tilde f,\tilde g)-H(0,1)|<\delta.
\]
Moreover, if we choose $\delta'>0$ sufficiently small, we can assume that $|u|>f_0$ for all
\[
	(u,v)\in B((f,g)(\overline{R}),\delta').
\]
By Lemma \ref{lemconvconssol}, there exist $\tilde R>0$ and $\tilde \gamma>0$ such that if $\rho\geq \tilde R$ and 
\[
	\|(\tilde f,\tilde g)-X_1\|<\tilde \gamma
\]
then
{\[
	\|(f_{\tilde f,\tilde g,\eta},g_{\tilde f,\tilde g,\eta})(\cdot+\rho)-(f,g)\|_{\infty,[0,\overline{R}]}<\delta'.
\]}
Since by hypothesis, there are $\epsilon\in(0,1)$ and for any $x\in(1-\epsilon,1)$ a constant $R_1>\max(R,\tilde{R})$ such that 
\[
	(\fxeta,\gxeta)(R_1)\in B(X_1, \tilde\gamma)\cap \mathcal{A}
\]
and 
$\gxeta$ changes sign exactly $k$ times on $[0,R_1],$  we get
\[
	\|(\fxeta,\gxeta)(\cdot+R_1)-(f,g)\|_{\infty,[0,\overline{R}]}<\delta'.
\]
In particular,
\[
	|(\fxeta,\gxeta)(R_1+\overline{R})-(f,g)(\overline{R})|<\delta'
\]
and $\gxeta$ changes sign exactly $k+1$ times on $[0,R_1+\overline{R}].$
\item
\textbf{Control of the solutions of \eqref{pbeq_nonautonomereg} in} $\mathcal{A}\cap\{(f,g)\in \mathbb R^2:~|f|>f_0\}.$
 Let  
$$
R_2:=\inf\{r>\overline R+R_1: \vb\fxeta(r)\vb < f_0\}.
$$
With the same arguments used in the proof of property $(H^0_1)$, we prove that 
\[
	(\fxeta,\gxeta)(R_2)\in B(X_2, \gamma)\cap \mathcal{A}
\]
and 
$\gxeta$ changes sign exactly $k+1$ times on $[0,R_2].$ We proved that
\[
	(H^k_1)\Rightarrow (H^{k+1}_2).
\]
\end{enumerate}
}
Thanks to the symmetry of the system, we also get
\[
(H^k_3)\Rightarrow (H^{k+1}_4).
\]
The proof of the remaining implications
\begin{align*}
	(H^k_2)\Rightarrow (H^{k}_3)\\
	(H^k_4)\Rightarrow (H^{k}_1)
\end{align*}
uses the same ideas.

%
%
%
%
%
%
%
%
%
%
%
\subsection{Proof of the remaining points of Lemma \ref{lem_proprieteAkIk}}
In this part, we assume that $\eta\in(0,\sqrt{a-b}-\sqrt{\frac{a}{2}})$ is fixed.

{ First of all, we remark that point \eqref{lem_Akouvert} follows directly from Lemma \ref{lemexistencesoleta}.} 

For the remaining points, we need the following preliminary lemma.

\begin{lem}\label{lem_trapping1}
There exists $c_0>0$ universal constant such that if 
\begin{enumerate}[(i)]
	\item $H(\fxeta,\gxeta)(R)< \frac{c_0}{R}$
	\item $\gxeta(R)\in\left(0,\sqrt{\frac{2b}{a}}\right)~\text{and}~\fxeta(R)<0~\\\text{or}~\gxeta(R)\in\left({ -}\sqrt{\frac{2b}{a}},0\right)~\text{and}~\fxeta(R)>0$,
	\item $\gxeta$ changes sign $k$ times on $[0,R];$
\end{enumerate} 
for $x\in(0,1),$ $R>0,$ $\eta\in(0,\sqrt{a-b}-\sqrt{\frac{a}{2}})$ and $k\in \mathbb{N}$,
then $x$ belongs to $A_k\cup I_k\cup A_{k+1}$. 
\end{lem}
\begin{proof}
We define
\[
	c_0:=\sqrt{\frac{2b}{9a(a-b)}}\frac{80b^2}{81a}.
\]
We can assume thanks to the symmetries of the system that 
\[
	H(\fxeta,\gxeta)(R)< \frac{c_0}{R},~\gxeta(R)\in\left(-\sqrt{\frac{2b}{a}},0\right)~\text{ and }~\fxeta(R)>0
\] 
for some $x\in (0,1)$ and $R>0.$
First of all, we remark that if there exists $\tilde R$ such that $H(\fxeta,\gxeta)(\tilde R)\leq 0$, then 
\[
	\varphi_\eta(\fxeta,\gxeta)(\tilde R)=1
\]
by Lemma \ref{lemposencorner}. Hence, by Lemma \ref{lem_deriveenergiereg} and Remark \ref{rem_decroissanceenergie}, we deduce
\[
	H(\fxeta,\gxeta)(r)<0
\] 
for all $r>\tilde R$. Since
\[
	\{(f,g):~g=0\}\cap H^{-1}(-\infty,0)=\emptyset,
\]
$\gxeta$ does not change sign anymore on $(\tilde R,+\infty)$. 
Moreover, if $\gxeta$ has no changes of sign in $(R,+\infty)$ then either 
\[
	\underset{r\rightarrow+\infty}{\lim}H(\fxeta,\gxeta)(r)<0,
\]
and $x\in A_k$ or
\[
	\underset{r\rightarrow+\infty}{\lim}H(\fxeta,\gxeta)(r)\geq0,
\]
and $x\in I_k$ by Lemma \ref{lem_tourexpdecay}.

We assume, by contradiction, that $x\notin A_k\cup I_k\cup A_{k+1}$ then
$\gxeta$ changes sign at least once in $(R,+\infty).$ 
Next, we denote
\[
	\overline{R} := \inf\{r>R:~\fxeta(r)\leq 0\}\in(R,+\infty].
\]
Since $\gxeta$ is increasing for all $r\in[R,\overline{R}]$, $\gxeta$ changes sign at most once before $(\fxeta,\gxeta)$ exits $\{(f,g):~f>0\}$. 
Moreover, we claim that $R<\overline{R}<+\infty$. 
Indeed, if $\overline{R}=+\infty$, $\gxeta$ changes sign $k$ or $k+1$ times on $\mathbb{R}^+$. Then, we have either
\[
	\underset{r\rightarrow+\infty}{\lim}H(\fxeta,\gxeta)(r)<0,
\]
and $x\in A_{k}\cup A_{k+1}$ or 
\[
	\underset{r\rightarrow+\infty}{\lim}H(\fxeta,\gxeta)(r)\geq0,
\]	
and $x\in I_k\cup I_{k+1}$ by Lemma \ref{lem_tourexpdecay}. Moreover, if $x\in I_{k+1}$, Lemma \ref{lem_tourexpdecay} ensures that $\gxeta$ decays exponentially to $0$ this contradicts the fact that $\gxeta$ is positive and increasing between 
\[
	\inf\{r\geq R:~\gxeta(r)\geq 0\}\in (R,\overline{R})
\]
and $\overline{R}.$ Nevertheless, we assumed that $x\notin A_k\cup I_k\cup A_{k+1}$ hence $\overline{R}<+\infty$.
As a consequence, we get
\[
	\fxeta(\overline{R})=0
\]
and, since $x\notin A_{k+1}$,  
\[
	H(\fxeta,\gxeta)(r)>0.
\]
for all $r\leq \overline{R}$.
Moreover, we have
\[
	\gxeta(\overline{R})>\sqrt{\frac{2b}{a}}
\]
since $H(0,x)\le 0$ for all $x\in[-\sqrt{\frac{2b}{a}},\sqrt{\frac{2b}{a}}]$.

Next, we denote 
\[
	R' :=\sup\{r\in(R,\overline{R}),~\gxeta(r)\leq \frac{1}{3}\sqrt\frac{2b}{a}\}
\]
and
\[
	R'':=\inf\{r>R:~\gxeta(r)\geq\frac{2}{3}\sqrt\frac{2b}{a}\},
\]
and we remark that this quantities are well-defined. For all $r\in(R,R''), $ we get
\begin{align*}
	\frac{d}{dr}H(\fxeta,\gxeta)(r)&=-\frac{2}{r}\fxeta^2(1-\gxeta^2)\\
	&=-\frac{4}{r}\left(H(\fxeta,\gxeta)(r)-\frac{a}{4}\gxeta(r)^4+\frac{b}{2}\gxeta^2(r)\right),
\end{align*}
and
\begin{align}\label{ineq_energietrapping}
	\frac{d}{dr}\left(r^4H(\fxeta,\gxeta)(r)\right) = r^3\gxeta^2(r)\left(a\gxeta^2(r)-2b\right)\leq 0
\end{align}
since $\gxeta^2(r)<\frac{2b}{a}$ for all $r\in[R,R''].$ Moreover,  we have 
\[
	\gxeta(r)\in\left[\frac{1}{3}\sqrt\frac{2b}{a},\frac{2}{3}\sqrt\frac{2b}{a}\right]
\]
for all $r\in [R',R'']$ and 
\begin{align}\label{ineq:gxeta}
	\frac{1}{3}\sqrt{\frac{2b}{a}}&=\gxeta(R'')-\gxeta(R') =\int^{R''}_{R'}~\fxeta(s)(1-\gxeta^2(s))ds\\
	\nonumber&\leq\sqrt{a-b}\,(R''-R').
\end{align}
 Integrating  inequality \eqref{ineq_energietrapping}, we have, thanks to inequality \eqref{ineq:gxeta},
\begin{align}\label{ineq:hxeta}
(R'')^4&H(\fxeta,\gxeta)(R'')-(R')^4H(\fxeta,\gxeta)(R')\leq -c_1\left((R'')^4 -(R')^4\right)\\
	\nonumber&\leq-c_1(R''-R')(R''^3+R''^2R'+R''R'^2+R'^3)\\
	\nonumber&\leq -4c_1\left(\sqrt{\frac{2b}{9a(a-b)}}\right)R^3= -c_0R^3
\end{align}
for $c_1=\frac{20b^2}{81a}$, since 
\[
	c_0=4c_1\sqrt{\frac{2b}{9a(a-b)}}=\sqrt{\frac{2b}{9a(a-b)}}\frac{80b^2}{81a}.
\]
Then, we obtain by inequalities \eqref{ineq_energietrapping} and \eqref{ineq:hxeta}
\begin{align*}
	(R'')^4H(\fxeta,\gxeta)(R'')&\leq -c_0R^3+(R')^4H(\fxeta,\gxeta)(R')\\
	&\leq R^4\left(-\frac{c_0}{R}+H(\fxeta,\gxeta)(R)\right)\\
	&<0.
\end{align*}
This is impossible since
\[
	H(\fxeta,\gxeta)(R'')>0.	
\]
\end{proof}
\subsubsection{ Proof of point (\ref{lem_trapping}) of Lemma \ref{lem_proprieteAkIk} } 

\begin{lem}\label{lem_trappingcontinuity}
Let $k\in \mathbb{N}$ and $ \eta\in(0,\sqrt{a-b}-\sqrt{\frac{a}{2}})$. If $x\in I_k$ then there is $\epsilon>0$ such that 
\[
	(x-\epsilon,x+\epsilon)\subset A_k\cup I_k\cup A_{k+1}.
\]
\end{lem}
\begin{proof}
By Lemma \ref{lem_tourexpdecay}, there exists $C,K>0$ such that  
\[
	(|\fxeta|+|\gxeta|)(r)\leq C\exp(-Kr)
\]
for all $r$ and $H(\fxeta,\gxeta)$ converges exponentially to $0$. We easily get that there is $R$ such that the assumptions of Lemma \ref{lem_trapping1} are fulfilled for $x$ at $R$. Then, by Lemma \ref{lemexistencesoleta}, there is $\epsilon>0$ such that for all $y\in (x-\epsilon,x+\epsilon)$, $g_{y,\eta}$ changes sign $k$ times on $[0,R],$
\[
	H(f_{y,\eta},g_{y,\eta})(R)<\frac{c_0}{R}
\]
and
\[
	f_{y,\eta}(R)<0,~g_{y,\eta}(R)\in(0,\sqrt{\frac{2b}{a}})~\text{or}~f_{y,\eta}(R)>0,~g_{y,\eta}(R)\in(-\sqrt{\frac{2b}{a}},0).
\]
Thus, by Lemma \ref{lem_trapping1}, we have that 
\[
	(x-\epsilon,x+\epsilon)\subset A_k\cup I_k\cup A_{k+1}.
\]
\end{proof}
\subsubsection{ Proof of point (\ref{lem_supAk}) of Lemma \ref{lem_proprieteAkIk} }

\begin{lem}\label{lem_supAk1}
Let $k\in \mathbb{N}$ and $\eta\in(0,\sqrt{a-b}-\sqrt{\frac{a}{2}}).$ If $A_k$ is non-empty, then 
\[
	\sup A_k\in I_k\cup I_{k-1}.
\]
\end{lem}
\begin{proof}
Thanks to points \eqref{lem_Akouvert}  and \eqref{lem_AkIkborne} of Lemma \ref{lem_proprieteAkIk}, 
\[
	x:=\sup A_k\in (0,1)\backslash \underset{n\in \mathbb{N}}{\cup}~A_n.
\]
Let $\{x_i\}\subset A_k$ be such that 
\[
 	\underset{i\rightarrow+\infty}{\lim}~x_i=x.
\]
Suppose by contradiction that $x\notin  \underset{n\in \mathbb{N}}{\cup}~(I_n\cup A_n),$ then in particular
\[
	H(\fxeta,\gxeta)(r)>0
\]
for all $r>0.$ Moreover, Lemma \ref{lem_tourexpdecay} ensures that $\gxeta$ changes sign an infinite number of times. 

Let $R>0$ be such that $\gxeta$ changes sign more than $k+1$ times in $[0,R]$ at $0<r_1<\dots<r_{k+1}.$ Hence, there is $\epsilon>0$ such that for all $j\in\{1,\dots,k+1\}$, all $r\in(0,\epsilon)$ we have
\[
	g_{x,\eta}(r_j-r)g_{x,\eta}(r_j+r)\leq0
\]
and
\[
	g_{x,\eta}(r_j-\epsilon)g_{x,\eta}(r_j+\epsilon)<0.
\]
Then, by Lemma \ref{lemexistencesoleta}, there is $M>0$ such that if $i\geq M$ then
\[
	g_{x_i,\eta}(r_j-\epsilon)g_{x_i,\eta}(r_j+\epsilon)<g_{x,\eta}(r_j-\epsilon)g_{x,\eta}(r_j+\epsilon)/2<0
\] for all $j\in\{1,\dots,k+1\}.$ Thus, for all $i\geq M$ and all $j\in\{1,\dots,k+1\}$,  there is a real number $r_j^i\in(r_j-\epsilon,r_j+\epsilon)$ such that $g_{x_{ i},\eta}(r^i_j)=0.$ Then, we get  $g_{x_{ i},\eta}'(r^i_j)\ne0$ so that $g_{x_j,\eta}$ changes sign more that $k+1$ times at the points $r^i_j.$ This is impossible because $x_{ i}\in A_k.$ Hence, we have that
\[
	\sup A_k\in I_m
\]
for some $m\in \mathbb{N}$ and by point \eqref{lem_trapping}, we get the result.
\end{proof}
\subsubsection{ Proof of point (\ref{lem_supIk}) of Lemma \ref{lem_proprieteAkIk} }
\begin{lem}
Let $k\in \mathbb{N}$ and $\eta\in(0,\sqrt{a-b}-\sqrt{\frac{a}{2}})$. If $I_k$ is non-empty, then 
\[
	\sup I_k\in I_k.
\]
\end{lem}
\begin{proof}
The proof follows the same ideas as the one of Lemma \ref{lem_supAk1}. We get that
\[
	\sup I_k \in I_j
\]
for some $j\in \mathbb{N}$ and by point \eqref{lem_trapping}, we get the result.
\end{proof}
\section{Proof of theorem \ref{thm_main}}\label{secproofth}
We give now the proof of Theorem \ref{thm_main} by taking the limit when $\eta$ tends to $0$. 
\begin{proof}
Let us fix $k\in \mathbb{N}.$ For all $\eta\in(0,\sqrt{a-b}-\sqrt{\frac{a}{2}}),$ there is  $x_\eta\in (0,1)$ such that  
\[
	\underset{r\rightarrow+\infty}{\lim}~ (f_{x_\eta,\eta},g_{x_\eta,\eta})(r)=(0,0),~g_{x_\eta,\eta} ~\text{has $k$ changes of sign on $(0,+\infty)$}
\]
by Proposition \ref{prop_shooting}.
We also know that
\[
	H(f_{x_\eta,\eta},g_{x_\eta,\eta})(r)> 0
\]
for all $r\in\mathbb{R}^+$. Since for all $x\in(0,\sqrt{\frac{2b}{a}})$ 
	\[
		H(0,x)<0,
	\]
we deduce that
	\[
		\{x_\eta\}_\eta\subset\left[\sqrt{\frac{2b}{a}},1\right).
	\]
	Thus, there is a subsequence $\{\eta_n\}_n$ such that
	\[\left\{\begin{array}{l}
		\underset{n\rightarrow +\infty}{\lim}~\eta_n= 0\\
		\underset{n\rightarrow +\infty}{\lim}~x_{\eta_n}=x_0\in \left[\sqrt{\frac{2b}{a}},1\right].
	\end{array}\right.\]
	By Lemma \ref{lemexistencesoleta}, we get  that for all $R>0, \epsilon>0,$ there exists $N>0$ such that if $n\geq N$ then
	\[
		\|(f_{x_0,0},g_{x_0,0})-(f_{x_{\eta_n},\eta_n},g_{x_{\eta_n},\eta_n})\|_{\infty,[0,R]}\leq \epsilon
	\]
	where $\|.\|_{\infty,[0,R]}$ is the uniform norm on the set $\mathcal{C}([0,R],\mathbb{R}^2).$
	Thus, $(f_{x_0,0},g_{x_0,0})$ is a solution of equations \eqref{eq_nonautonome} such that
	\[\left\{\begin{array}{l}
		(f_{x_0,0},g_{x_0,0})(0)=(0,x_0)\\
		H(f_{x_0,0},g_{x_0,0})(r)\geq 0,~\text{for all}~r\in \mathbb{R}^+
	\end{array}\right.\]
	and Remark \ref{rem_decroissanceenergie} ensures that $x_0\in\left(\sqrt{\frac{2b}{a}},1\right]$. 
	
	To conclude, we have to show now that $x_0\in\left(\sqrt{\frac{2b}{a}},1\right)$ since $(f_{1,0},g_{1,0})$ is not a localized solution of \eqref{eq_nonautonome}. 
	
	Assume, by contradiction, that $x_0=1$; then
	\[
		H(f_{x_0,0},g_{x_0,0})(r)=H(0,1)
	\]
	for all $r\geq0.$ We denote $H_0:= H(0,1)/2>0$ and we define
	\[
		R_n := \inf\{r>0:~H(f_{x_{\eta_n},\eta_n},g_{x_{\eta_n},\eta_n})(r)\leq H_0\}\in(0,+\infty).
	\]
	We have
	\[
		H(f_{x_{\eta_n},\eta_n},g_{x_{\eta_n},\eta_n})(R_n)=H_0.
	\]
	We claim that $\{R_n\}_n$ tends to $+\infty.$ Indeed, if there is a subsequence also denoted $\{R_n\}_n$ and a real number $R>0$ such that $R_n\in[0,R]$ then, we get the following contradiction
	\[
		H(0,1) = \underset{n\rightarrow +\infty}{\lim}~H(f_{x_{\eta_n},\eta_n},g_{x_{\eta_n},\eta_n})(R)\leq H_0.
	\]
	 Next, since $\mathcal{A}$ is a compact set,  there is $(f_0,g_0)\in \mathcal{A}$ such that up to extraction,
	\[
		\underset{n\rightarrow +\infty}{\lim}~(f_{x_{\eta_n},\eta_n},g_{x_{\eta_n},\eta_n})(R_n)=(f_0,g_0)
	\]
	and
	\[
		H(f_0,g_0)=H_0.
	\]
	We denote by $T>0$ the period of the solution of the Hamiltonian system of equations \eqref{eq_autonome} of energy equal to $H_0$. Let us consider now the following Cauchy problem
	\begin{equation}\label{pbeq_nonautonomereginf}
	\left\{
		\begin{array}{rl}
			f'+\frac{2\varphi_\eta(f,g)}{r+\rho}f	&=g(f^2-ag^2+b),\\
			g'			&=f(1-g^2),\\
			(f,g)(0)&=X. 
		\end{array}
	\right.
	\end{equation}
	Its solutions depend continuously on the parameters $(X,\rho,\eta)$ on every interval $[0,R] $ just as in Lemma \ref{lemexistencesoleta}. So, $(f_{x_{\eta_n},\eta_n},g_{x_{\eta_n},\eta_n})(~.+R_n)$ tends uniformly on $[0,\frac{(k+2)T}{2}]$ to a solution $(f,g)$ of  
		\begin{equation*}
	\left\{
		\begin{array}{rl}
			f'	&=g(f^2-ag^2+b),\\
			g'			&=f(1-g^2),\\
			(f,g)(0)&=(f_0,g_0). 
		\end{array}
	\right.
	\end{equation*}
	Moreover, $(f,g)$ is periodic of period $T$ which implies that $g$ has at least $k+1$ changes of sign on $\left[0,\frac{(k+2)T}{2}\right]$ \emph{i.e.} there is $\epsilon>0$, and $0<r_1<\dots<r_{k+1}<\frac{(k+2)T}{2}$ such that
	\[
		g(r_i+r)g(r_i-r)\leq0
	\]
	for all $r\in(0,\epsilon)$
	and
	\[
		g(r_i+\epsilon)g(r_i-\epsilon)<0
	\]	
	 for all $i\in\{1,\dots k+1\}$. 
	 
	 Hence, there is $N>0$ such that for all $n>N$ and all  $i\in\{1,\dots k+1\},$ we get
	\[
		g_{x_{\eta_n},\eta_n}(R_n+r_i+\epsilon)g_{x_{\eta_n},\eta_n}(R_n+r_i-\epsilon)<g(r_i+\epsilon)g(r_i-\epsilon)/2<0,
	\]
	which implies that there is $r^i_n\in (r_i-\epsilon,r_i+\epsilon)$ such that
	\[
		g_{x_{\eta_n},\eta_n}(R_n+r^i_n)=0.
	\]
	As a conclusion, since $(f_{x_{\eta_n},\eta_n},g_{x_{\eta_n},\eta_n})$ is a solution of \eqref{eq_nonautonomereg}, $g_{x_{\eta_n},\eta_n}'((R_n+r^i_n))\ne0$ and $g_{x_{\eta_n},\eta_n}$  has at least $k+1$ changes of sign at the points $r^i_n.$ This is impossible because $g_{x_{\eta_n},\eta_n}$  has exactly $k$ changes of sign. As a consequence $x_0<1$. 
	
	Moreover, with the same arguments used above, we prove that $g_{x_{0},0}$ changes sign a finite number of times $k_0$. Hence, Lemma \ref{lem_tourexpdecay} ensures that $(f_{x_0,0},g_{x_0,0})$ converge exponentially to $(0,0)$.
 
 As a consequence, there is $R>0$ such that $g_{x_0,0}$ changes sign $k_0$ times in $[0,R]$, 
\[
	H(f_{x_0,0},g_{x_0,0})(R)<\frac{c_0}{R}
\]
and
\[
	f_{x_0,0}(R)<0,~g_{x_0,0}(R)\in(0,\sqrt{\frac{2b}{a}})~\text{or}~f_{x_0,0}(R)>0,~g_{x_0,0}(R)\in(-\sqrt{\frac{2b}{a}},0).
\]
Hence, there is $\epsilon\in(0,1)$ such that for all $x\in(x_0-\epsilon,x_0+\epsilon)$ and all $\eta\in(0,\epsilon),$ $g_{x,\eta}$ changes sign $k_0$ of times in $[0,R]$,
\[
	H(f_{x,\eta},g_{x,\eta})(R)<\frac{c_0}{R}
\]
and
\[
	f_{x,\eta}(R)<0,~g_{x,\eta}(R)\in(0,\sqrt{\frac{2b}{a}})~\text{or}~f_{x,\eta}(R)>0,~g_{x,\eta}(R)\in(-\sqrt{\frac{2b}{a}},0).
\]
By applying Lemma \ref{lem_trapping1}, we get, for $\eta\in (0,\epsilon)$ fixed, 
\[
	(x_0-\epsilon,x_0+\epsilon)\subset A_{k_0}\cup I_{k_0}\cup A_{k_0+1}.
\]
Remark that the set $A_k$ and $I_k$ depends on $\eta$, hence, it is important to fix $\eta$ before writing such a property. By definition of $\{x_{\eta_n}\}_n$, there is $N\in\mathbb{N}$ such that for all $n\geq N$
\[
	x_{\eta_n}\in(x_0-\epsilon,x_0+\epsilon),~\eta_n\in(0,\epsilon).
\]
As a conclusion, for $n\geq N$ fixed $x_{\eta_n}\in I_{k_0}$. This ensures that $k=k_0$.

Finally, by Remark \ref{remequivsignvanish}, if $g_{x_0,0}$ changes sign $k$ times on $(0,+\infty)$ then $g_{x_0,0}$ has $k$ zeros on $(0,+\infty)$. To conclude, it remains to prove that $f_{x_0,0}$ has $k$ zeros on $(0,+\infty)$ as well.

Let $\{r_1,\ldots,r_k\}$ be the zeros of $g_{x_0,0}$ on $(0,+\infty)$. First of all, we prove by induction that, for all $i=1\ldots,k$, $f_{x_0,0}$ has $i-1$ zeros on $(0,r_i)$. This property is true for $i=1$. Indeed, suppose by contradiction that $f_{x_0,0}(\bar r)=0$ for some $\bar r\in(0,r_1)$. Hence, using the first equation of \eqref{eq_nonautonome}, we get $H(f_{x_0,0}, g_{x_0,0})(\bar r)<0$. That is impossible.  
Next, if the property holds true for $i-1$, then it holds true for $i$. Indeed, suppose that $f_{x_0,0}$ has $i-2$ zeros on $(0,r_{i-1})$; we prove that $f_{x_0,0}$ has $1$ zero on $(r_{i-1},r_i)$. By contradiction, if $f_{x_0,0}$ does not change sing on $(r_{i-1},r_i)$, the second equation of  \eqref{eq_nonautonome} implies that $g_{x_0,0}$ is monotone on $(r_{i-1},r_i)$. This contradicts the fact that $g_{x_0,0}$ changes sign at $r_{i-1}$ and $r_i$. Hence, $f_{x_0,0}$ has at least $1$ zero at $\bar r\in (r_{i-1},r_i)$. Now suppose that there exists $\tilde r \in (\bar r ,r_i)$ such that $f_{x_0,0}(\tilde r)=0$. One of the following situations arise: $g(r)<0$ and $f(r)>0$ on $(\bar r ,\tilde r)$ or  $g(r)>0$ and $f(r)<0$ on $(\bar r , \tilde r)$. Using again the first equation of \eqref{eq_nonautonome}, in both cases we get $g_{x_0,0}^2(\tilde r )\le \frac{b}{a}$ which implies $H(f_{x_0,0}, g_{x_0,0})(\tilde r)<0$, a contradiction. Hence, for all $i=1\ldots,k$, $f_{x_0,0}$ has $i-1$ zeros on $(0,r_i)$.
Finally, using the same arguments, we show that $f_{x_0,0}$ has $1$ zero on $(r_k,+\infty)$ and we conclude that $f_{x_0,0}$ has $k$ zeros on $(0,+\infty)$.
		
\end{proof}

\appendix
\section{Geometric properties of $H$.}
%
%
%
%
%
We remind that  $\mathcal{A}=\{(f_0,g_0)\in\mathbb{R}^2|\, 2f_0^2-ag_0^2-(a-2b)\le 0,\; g_0^2\le 1\}$ is the set of admissible points. Let us remark that $(f,g)\in\mathbb{R}^2$ satisfies $H(f,g)=H(0,1)=\frac{1}{4}(a-2b)$ if and only if
\begin{align*}
	0&=H(f,g)-H(0,1) = \frac{1}{2}\left(f^2(1-g^2)+\frac{a}{2}(g^4-1)-b(g^2-1)\right)\\
	& =  \frac{(1-g^2)}{2}\left(f^2-\frac{a}{2}(g^2+1)+b\right)
\end{align*}
i.e. if and only if $g^2=1$ or $f^2-\frac{a}{2}(g^2+1)+b=0.$ $\mathring{\mathcal{A}}$ is the connected component of $\{(f,g):~H(f,g)\ne H(0,1)\}$ which contains $(0,0)$. Thanks to the symmetries of $H$, we can restrict our study of $\mathcal{A}$ to the set $\{(f,g):~ f\geq 0,\,g\geq 0\}.$
{
\begin{lem}\label{lem_EtudeReg}
Let $f_0\in \left(\sqrt{\frac{a-2b}{2}},\sqrt{a-b}\right)$ and $E_c=H\left(f_0,\sqrt{\frac{f_0^2+b}{a}}\right)$. Then there exist two monotone continuous functions
\begin{align*}
G_{1}: [E_c,H(0,1)]&\to \left[\sqrt{\frac{f_0^2+b}{a}},1\right]\\
G_{2}: [E_c,H(0,1)]&\to \left[\sqrt{\frac{2}{a}(f_0^2+b)-1},\sqrt{\frac{f_0^2+b}{a}}\right]
\end{align*}
such that, for $i=1,2$ and for all $E\in [E_c,H(0,1)]$,
\begin{align*}
&H(f_0,G_i(E))=E,\\
&\{(f_0, G_1(E)),(f_0, G_2(E))\} = \{(f_0,g)\in \mathcal{A}:g\geq 0\}\cap H^{-1}(\{E\}),
\end{align*}
and $G_1(E)\ge G_2(E)$.
\end{lem}
%
%
%
%
%

\begin{proof}
First of all, we observe that  
\begin{align*}
&\mathcal A \cap (\mathbb R_+)^2=\{(f,g)\in [0,\sqrt{a-b}]\times[0,1]: ag^2\ge 2f^2-(a-2b)\}\\
&{=\left(\left[0,\sqrt{\frac{a-2b}{2}}\right]\times[0,1]\right)}\\
&\cup \left\{(f,g):f\in\left(\sqrt{\frac{a-2b}{2}},\sqrt{a-b}\right],g\in  \left[\sqrt{\frac{2f^2-(a-2b)}{a}}, 1\right]\right\}
\end{align*}
Next, let $f_0\in\left(\sqrt{\frac{a-2b}{2}},\sqrt{a-b}\right)$ fixed, and define the function 
\begin{align*}
G:\left[\sqrt{\frac{2f_0^2-(a-2b)}{a}},1\right]&\to \mathbb R\\
g&\mapsto H(f_0,g).
\end{align*}
Since $G'(g) = g(-b+ag^2-f_0^2)$, we deduce that  $G$ is continuous, increasing in $\left[\sqrt{\frac{f_0^2+b}{a}},1\right]$ and decreasing in $\left[\sqrt{\frac{2f_0^2-(a-2b)}{a}},\sqrt{\frac{f_0^2+b}{a}}\right]$. Note that
\[
	\sqrt{\frac{2f_0^2-(a-2b)}{a}}<\sqrt{\frac{f_0^2+b}{a}}
\]
since $f_0<\sqrt{a-b}$. Moreover, we have
\[
	G\left(\sqrt{\frac{2f_0^2-(a-2b)}{a}}\right)=G(1)=H(0,1)
\]
and we define $E_c:=G\left(\sqrt{\frac{f_0^2+b}{a}}\right)=H\left(f_0,\sqrt{\frac{f_0^2+b}{a}}\right)$.
As a consequence, $G$ is a one-to-one function from $\left[\sqrt{\frac{f_0^2+b}{a}},1\right]$ onto $[E_c,H(0,1)]$ whose inverse is continuous and is denoted by $G_1$. Therefore, 
\begin{align*}
G_{1}: [E_c,H(0,1)]&\to \left[\sqrt{\frac{f_0^2+b}{a}},1\right]\\
E&\mapsto G_1(E)
\end{align*}
is a continuous function such that $H(f_0,G_1(E))=E$.
Similarly, we denote by $G_2$ the inverse of the restriction of the function $G$ to the set ${\left[\sqrt{\frac{2}{a}(f_0^2+b)-1},\sqrt{\frac{f_0^2+b}{a}}\right]}$. Hence, 
\begin{align*}
G_{2}: [E_c,H(0,1)]&\to \left[\sqrt{\frac{2}{a}(f_0^2+b)-1},\sqrt{\frac{f_0^2+b}{a}}\right]\\
E&\mapsto G_2(E)
\end{align*}
is a continuous function such that $H(f_0,G_2(E))=E$. Moreover, $G_1(E)\ge G_2(E)$ for all $E\in [E_c,H(0,1)]$.
\end{proof}

\begin{lem}\label{lemposencorner} Let $f_0\in \left(\sqrt{\frac{a}{2}},\sqrt{a-b}\right)$. If
$$
(\overline f,\overline g)\in \mathcal {A} \cap\{(f,g)\in \mathbb R^2:\vb f\vb \ge f_0\}, 
$$
then $H(\overline f,\overline g)>0$.
\end{lem}

\begin{proof}
First of all, since $ f_0 >\sqrt\frac{a}{2}>\sqrt\frac{a-2b}{2}$, we have for
\begin{align*}
\lefteqn{(\overline f,\overline g)\in\mathcal A \cap \{(f,g)\in \mathbb R^2:\vb f\vb \ge f_0\}}\\
&&=\{(f,g)\in [f_0,\sqrt{a-b}]\times[-1,1]: ag^2\ge 2f^2-(a-2b)\},
\end{align*}
that
$$
\frac{2b}{a}<\frac{2\overline f^2-(a-2b)}{a}\le \overline g^2\le 1.
$$
On the other hand, if $g^2\le 1$ and $H(f,g)\le 0$, then $g^2\le \frac{2b}{a}$.
As a consequence, 
we get $H(\overline f,\overline g)>0$.
\end{proof}

}

\subsection*{Acknowledgment}This work was partially supported by the Grant ANR-10-BLAN 0101 of the French Ministry of research. The authors would like to thank Maria J. Esteban for the helpful discussions.

\end{document}